\documentclass[12pt,a4paper,dvipdfmx]{article}
\usepackage{amsmath,amsthm,amssymb,mathrsfs,enumerate,hyperref}
\usepackage[numbers]{natbib}
\def\A{\mathcal{A}}
\def\B{\mathcal{B}}
\def\M{\mathcal{M}}
\def\BB{I\!\!B}
\def\NN{I\!\!N}
\def\RR{I\!\!R}
\theoremstyle{plain}
\newtheorem{thm}{Theorem}[section]
\newtheorem{cor}{Corollary}[section]
\newtheorem{prop}{Proposition}[section]
\theoremstyle{definition}
\newtheorem{assmp}{Assumption}[section]
\newtheorem{dfn}{Definition}[section]

\numberwithin{equation}{section}
\allowdisplaybreaks

\def\hat{\widehat}
\def\emp{\emptyset}

\def\ox{\overline{x}}
\def\oy{\overline{y}}

\def\disp{\displaystyle}

\def\Limsup{\mathop{{\rm Lim}\,{\rm sup}}}

\def\tto{\;{\lower 1pt\hbox{$\rightarrow$}}\kern-10pt
\hbox{\raise 2pt\hbox{$\rightarrow$}}\;}
\def\Hat{\widehat}

\def\Bar{\overline}
\def\ra{\rangle}
\def\la{\langle}

\def\R{\mathbb{R}}
\def\ox{\bar{x}}
\def\oy{\bar{y}}

\def\substack#1#2{{\scriptstyle{#1}\atop\scriptstyle{#2}}}

\def\dn{\downarrow}

\def\ph{\varphi}
\def\emp{\emptyset}
\def\st{\stackrel}
\def\oR{\Bar{\R}}

\def\al{\alpha}

\def\th{\theta}

\def\beq{\begin{equation}}
\def\eeq{\end{equation}}

\title{Subdifferentials of Value Functions in Nonconvex Dynamic Programming for Nonstationary Stochastic Processes}
\date{\today}
\author{Boris S.\,Mordukhovich\thanks{Research of this author was partly supported by the US National Science Foundation under grants DMS-1512846 and DMS-1808978, by the US Air Force Office of Scientific Research under grant \#15RT0462, and by the Australian Research Council Discovery Project DP-190100555.}\\
{\small Department of Mathematics, Wayne State University}\\
{\small Detroit, MI 48202, USA}\\
{\footnotesize e-mail: boris@math.wayne.edu} \and
\\
Nobusumi Sagara\thanks{Research of this author was partly supported by JSPS KAKENHI Grant Number JP18K01518 from the Ministry of Education, Culture, Sports, Science and Technology, Japan.} \\
{\small Department of Economics, Hosei University}\\
{\small 4342, Aihara, Machida, Tokyo, 194-0298, Japan}\\
{\footnotesize e-mail: nsagara@hosei.ac.jp}}
\begin{document}
\maketitle
\setcounter{page}{0}
\thispagestyle{empty}
\clearpage

\begin{abstract}
The main goal of this paper is to apply the machinery of variational analysis and generalized differentiation to study infinite horizon stochastic dynamic programming (DP) with discrete time in the Banach space setting without convexity assumptions. Unlike to standard stochastic DP with stationary Markov processes, we investigate here stochastic DP in $L^p$ spaces to deal with nonstationary stochastic processes, which describe a more flexible learning procedure for the decision-maker. Our main concern is to calculate generalized subgradients of the corresponding value function and to derive necessary conditions for optimality in terms of the stochastic Euler inclusion under appropriate Lipschitzian assumptions. The usage of the subdifferential formula for integral functionals on $L^p$ spaces allows us, in particular, to find verifiable conditions to ensure smoothness of the value function without any convexity and/or interiority assumptions.\\


\end{abstract}
\clearpage

\section{Introduction}
This paper aims at applying advanced tools of variational analysis and generalized differentiation to investigate infinite horizon stochastic dynamic programming (DP) models with discrete time in general  Banach spaces without conventional convexity assumptions. Unlike to standard stochastic DP with stationary Markov processes studied in \cite{ae87,bs78,ms18,rzs09,sl89}, we consider here stochastic DP in $L^p$ spaces defined on arbitrary Banach spaces to deal with {\em nonstationary stochastic processes} in order to design a more flexible learning procedure for the decision-maker. It is well known in the literature on optimal economic growth under uncertainty that the stationarity of stochastic processes and the convexity of technologies and preferences are indispensable conditions to characterize optimal stationary programs and to establish the turnpike property; see \cite{ae87,be81,da74,ev74,je74,ra73,zi76,ya89}. Since the (stochastic) turnpike property is beyond the scope of this paper, we assume neither stationarity nor convexity for our stochastic nonstationary DP model.

It has been fully understood in dynamic optimization that value functions (or marginal functions in another terminology) play a crucial role in characterizing optimality along with other important variational properties. Since value/marginal functions are generally nondifferentiable in standard senses, the usual way of applications of value functions to the study optimality is through calculating their appropriate subdifferentials (collections of subgradients), which is not a simple task in structural models that arise in applied areas. We proceed here in this way to derive necessary optimality conditions in terms of a stochastic Euler inclusion in nonstationary DP models with Lipschitzian data in Banach spaces.

To incorporate the nonstationarity of stochastic processes, we follow the probabilistic specification in \cite{pa94,ta92} to embed stochastic DP into deterministic DP in the extended Banach space setting. The approach to reduce stochastic DP to deterministic DP was also developed in \cite{ya89} for optimal economic growth models with finite-dimensional commodity spaces and stationary stochastic processes under smoothness assumptions. We provide now a general framework for deterministic DP based on our preceding publication \cite{ms18} to incorporate infinite-dimensional commodity spaces for possible economic applications. As well known, necessary optimality conditions in terms of the (stochastic) Euler inclusions for convex models amount to the existence of a support price system along the lines of \cite{mc86,pa94,we73}. Our nonconvex stochastic DP model is essentially more involved. The necessary optimality conditions derived below by employing subdifferentiation of integral functionals in $L^p$ give us, in particular, efficient conditions for smoothness (more precisely, strict differentiability) of the value function in the model under consideration without any convexity and interiority assumptions. The obtained conditions for smoothness significantly improve the previously known results in this directions, which have always been of strong interest in economic modeling; see \cite{alv98,bs79,blv96,rzs09} and the references therein.\vspace*{0.05in}

The rest of the paper is organized as follows. Section~2 presents the necessary background and preliminary results from variational analysis and generalized differentiation broadly used below. In Section~3 we describe a deterministic DP model in Banach spaces governed by a discrete-time dynamic system with Lipschitzian data and derive necessary conditions for optimal solutions under rather mild assumptions by employing subdifferential calculus. Section~4 is devoted to the nonstationary stochastic DP model of our main interest here and derives necessary optimality conditions for them by using subdifferentiation of integral functionals. We conclude the paper in Section~5 by formulating some open questions.

\section{Preliminaries from Variational Analysis}

We split this section into 2 subsections. The first one contains definitions of the major constructions of generalized differentiation in variational analysis used in the paper. The second subsection is devoted to evaluating subgradients for a general class of marginal/value functions.

\subsection{Derivatives and Subdifferentials}

We begin with the generalized differential constructions by Clarke \cite{cl83} for Lipschitz continuous functions on arbitrary Banach spaces. Let $(E,\|\cdot\|)$ be a Banach space with its dual $E^*$, and let $\langle\cdot,\cdot\rangle$ signifies the dual system $\langle\cdot,\cdot\rangle$ on $E^*\times E$. Given an extended-real-valued function $\phi\colon E\to\overline{\RR}:=(-\infty,\infty]$ that is locally Lipschitzian around $\ox$, recall first its {\em generalized directional derivative} at $\ox$ in the direction $h\in E$ defined by
\begin{equation}\label{dd}
\phi^\circ(\bar{x};h):=\limsup_\substack{x\to\ox}{\theta\downarrow 0}\frac{\phi(x+\theta h)-\phi(x)}{\theta}.
\end{equation}
A crucial property of the function $h\mapsto\phi^\circ(\bar{x};h)$ is its automatic convexity, which is the source---together with the convex separation theorem---of nice calculus rules for it as well as for the {\em generalized gradient} (known also as the convexified or Clarke subdifferential) of $\phi$ at $\ox$ induced by \eqref{dd} via the conventional duality scheme
\begin{equation}\label{gg}
\partial^\circ\phi(\bar{x}):=\big\{x^*\in E^*\big|\;\langle x^*,h\rangle\le\phi^\circ(\bar{x};h)\;\mbox{ for every }\;h\in E\big\}
\end{equation}
of generating subdifferentials from directional derivatives. It is easy to observe that the set $\partial^\circ\phi(\bar{x})$ is nonempty, convex, and $\mathit{w}^*$-\hspace{0pt}compact in $E^*$. Furthermore, the convexity of $\phi^\circ(\ox;\cdot)$ easily implies by convex separation that \eqref{dd} is the support function of the generalized gradient, i.e., we have
$$
\phi^\circ(\bar{x};h)=\max_{x^*\in \partial^\circ\phi(\bar{x})}\langle x^*,h \rangle\quad\text{for every $h\in E$}.
$$
Recall that the function $\phi\colon E\to\overline{\RR}$ is (directionally) {\em regular} at $\ox$ if the classical directional derivative
$$
\phi'(\bar{x};h):=\lim_{\theta\downarrow 0}\frac{\phi(\bar{x}+\theta h)-\phi(\bar{x})}{\theta}
$$
exists and agrees with \eqref{dd}, i.e., $\phi'(\bar{x};h)=\phi^\circ(\bar{x};h)$ for every $h\in E$. The class of regular functions contains smooth and convex ones as well as their reasonable extensions and compositions; see \cite{cl83} for the facts reviewed above. Recall that a function $\phi\colon E\to\overline{\RR}$ is \textit{strictly differentiable} at $\bar{x}$ with its strict derivative $\nabla\phi(\bar{x})\in E^*$ if $\phi(\bar{x})<\infty$ and
\[
\lim_\substack{h\to 0}{x\to\bar{x}}\frac{\phi(x+h)-\phi(x)-\langle\nabla\phi(\bar{x}),h\rangle}{\|h\|}=0.
\]
This property lies properly between the usual Fr\'echet differentiability of a function at the given point and its continuous differentiability in a neighborhood of the point. Note that strict differentiability of $\phi$ implies that $\phi$ is locally Lipschitzian around $\ox$ and regular at this point with $\partial^\circ\phi(\bar{x})=\{\nabla\phi(\ox)\}$; see \cite[Propositions~2.2.4 and 2.3.6]{cl83}.

The construction of the {\em generalized/Clarke normal cone} \cite{cl83} to a subset $C$ of $E$ is defined with the usage of the $w^*$-closure operation by
\begin{equation}\label{cnc}
N^\circ(\bar{x};C):=\overline{\bigcup_{\al>0}\al\,\partial^\circ d_C(\bar{x})}^{\,w^*} \quad\text{at }\;\bar{x}\in C
\end{equation}
via the generalized gradient of the Lipschitz continuous {\em distance function} to $C$ given by $d_C(x):=\inf_{c\in C}\|c-x\|$. Recall also that the \textit{generalized/Clarke tangent cone} to $C$ is defined by
$$
T^\circ(\bar{x};C):=\left\{x\in E\;\big|\;\langle x^*,x\rangle\le 0 \ \text{for every }\;x^*\in N^\circ(\bar{x};C)\right\}\quad\text{at $\bar{x}\in C$},
$$
which admits another representation $T^\circ(\bar{x};C)=\left\{h\in E\mid d_C^\circ(\bar{x};h)=0\right\}$. The (Bouligand-Severi) \textit{contingent cone} to $C$ is defined by
$$
K(\bar{x};C):=\left\{h\in E\;\left|\;\liminf_{\theta\downarrow 0}\frac{d_C(\bar{x}+\theta h)}{\theta}=0\right.\right\} \quad \text{at $\bar{x}\in C$}.
$$
It follows from the definition that $T^\circ(\bar{x};C)\subset K(\bar{x};C)$, but $K(\bar{x};C)$ may not be convex in contrast to $T^\circ(\bar{x};C)$ and $N^\circ(\bar{x};C)$. The set $C$ is (tangentially) {\em regular} at $\bar{x}\in C$ if $T^\circ(\bar{x};C)=K(\bar{x};C)$.\vspace*{0.05in}

We proceed further with alternative constructions of generalized differentiation, we refer the reader to the book by Mordukhovich \cite{mo06}; see also \cite{mo18,rw98} for the related and complementary material. Given an extended-real-valued function $\phi\colon E\to\oR$ and $\varepsilon\ge 0$, the $\varepsilon$-\textit{subdifferential} of $\phi$ at a point $\bar{x}\in E$ with $\phi(\ox)<\infty$ is defined by
\begin{equation}\label{fs}
\hat{\partial}_\varepsilon\phi(\bar{x}):=\left\{x^*\in E^* \left|\;\liminf_{x\to \bar{x}}\frac{\phi(x)-\phi(\bar{x})-\langle x^*,x-\bar{x} \rangle}{\| x-\bar{x}\|}\ge -\varepsilon \right.\right\}.
\end{equation}
When $\varepsilon=0$, we set $\hat{\partial}\phi(\bar{x}):=\hat{\partial}_0\phi(\bar{x})$ is called the \textit{regular subdifferential} of $\phi$ at $\bar{x}$ and is also known as the Fr\'echet or viscosity subdifferential, as well as the presubdifferential of $\phi$ at $\ox$. Then the {\em limiting subdifferential} (known also as the basic, general, or Mordukhovich one) of $\phi$ at $\ox$ is defined by
\begin{equation}\label{bs}
\partial\phi(\bar{x}):=\Limsup_{\begin{subarray}{c}x\stackrel{\phi}{\to}\bar{x} \\ \varepsilon\downarrow 0\end{subarray}}\hat{\partial}_\varepsilon\phi(\bar{x}),
\end{equation}
where the notation ``$\Limsup$" for a set-valued mapping/multifunction $\Psi\colon E\rightrightarrows E^*$ stands for the (Painlev\'{e}--Kuratowski) {\em sequential outer limit} defined by
$$
\Limsup_{x\to \bar{x}}\Psi(x):=\left\{x^*\in E^*\left|\begin{array}{l}\exists \text{ sequences } x_k\to\ox,\;x^*_k\st{w^*}{\to}x^* \ \text{with}\\x_k^*\in \Psi(x_k)\text{ for each $k=1,2,\dots$} \end{array} \right.\right\},
$$
and where the symbol $x\st{\phi}{\to}\ox$ means that $x\to\ox$ with $\phi(x)\to\phi(\ox)$.

Recall that a Banach space $E$ is {\em Asplund} if every convex continuous function $\phi\colon U\to\R$  defined on an open convex set $U\subset E$ is Fr\'echet differentiable on a dense subset of $U$. This class of Banach spaces is sufficiently large including, in particular, any space with a Fr\'echet differentiable bump function (hence any space admitting an equivalent norm Fr\'echet differentiable off the origin, i.e., a {\em Fr\'echet smooth renorm}, and therefore every reflexive space), any space with a separable dual, and any space $E$ whose dual space $E^*$ is {\em weakly compactly generated} meaning that there exists a weakly compact subset of $E^*$ whose linear span in norm sense. There are many useful characterizations of Asplund spaces; among the most remarkable ones we mention that $E$ is Asplund if and only if every closed separable subspace of $E^*$ has a separable dual. It is also relevant to mention that any separable Asplund space admits a Fr\'echet smooth renorming.

If $E$ is an Asplund space and $\phi$ is lower semicontinuous around $\ox$, then
$$
\partial\phi(\bar{x})=\Limsup_{x\stackrel{\phi}{\to}\bar{x}}\hat{\partial}\phi(\bar{x}),
$$
and hence \eqref{bs} has the following representation:
\begin{equation*}
\partial\phi(\ox)=\left\{x^*\in E^*\left|\begin{array}{ll}\exists\mbox{ sequences } x_k\to\ox,\;x^*_k\st{w^*}{\to}x^* \mbox{ such that}\\
\disp\liminf_{x\to x_k}\frac{\phi(x)-\phi(x_k)-\la x^*_k,x-x_k\ra}{\|x-x_k\|}\ge 0
\end{array}
\right.\right\}.
\end{equation*}

Similarly to but a bit differently from \eqref{cnc}, define the {\em basic/limiting normal cone} \cite{mo06} to a subset $C$ of a Banach space $E$ by
\begin{equation}
\label{lnc}
N(\bar{x};C):=\bigcup_{\al>0}\al\,\partial d_C(\bar{x};C) \quad\text{at } \bar{x}\in C
\end{equation}
via the limiting subdifferential \eqref{bs} of the distance function. Then we have by \cite[Theorem~3.57(i)]{mo06} that $N^\circ(\bar{x};C)=\overline{{\rm co}}^{\,w^*}N(\bar{x};C)$ provided that the space $E$ is Asplund and that $C$ is locally closed around $\bar{x}$, where the symbol $\overline{{\rm co}}^{\,w^*}$ signifies for the weak$^*$ topological closure in $E^*$ of the convex hull of the set in question. Respectively, \cite[Theorem~3.57(ii)]{mo06} tells us that if $\phi$ is locally Lipschitzian around $\ox$ on a Banach space $E$, then $\partial^\circ\phi(\ox)=\overline{{\rm co}}^{\,w^*}\partial\phi(\ox)$.

Finally in this subsection, recall that for any $\varepsilon\ge 0$ the \textit{$\varepsilon$-coderivative} of a set-valued mapping $\Gamma\colon E\rightrightarrows F$ at $(x,y)\in E\times F$ is the multifunction $\hat{D}^*_\varepsilon \Gamma(x,y)\colon F^*\rightrightarrows E^*$ given by
$$
\hat{D}^*_\varepsilon \Gamma(x,y)(y^*):=\left\{ x^*\in E^*\left|\ (x^*,-y^*)\in \hat{N}_\varepsilon((x,y);\mathrm{gph}\,\Gamma)\right.\right\},
$$
where $\mathrm{gph}\,\Gamma:=\{(x,y)\in E\times F\mid y\in \Gamma(x)\}$ is the graph of $\Gamma$, and where the $\varepsilon$-normal cone $\Hat N_\varepsilon$ is defined via the $\varepsilon$-subdifferential \eqref{fs} of the set indicator function equal $0$ at set points and $\infty$ otherwise. When $\varepsilon=0$, we set $\hat{D}^*\Gamma(x,y)(y^*):=\hat{D}^*_0\Gamma(x,y)(y^*)$, which is called the (Fr\'echet) \textit{regular coderivative} of $\Gamma$ at $(x,y)$. The (limiting, Mordukhovich) \textit{normal coderivative} of $\Gamma$ at $(\bar{x},\bar{y})\in E\times F$ is the multifunction $D^*_N\Gamma(\bar{x},\bar{y}):F^*\rightrightarrows E^*$ defined by
$$
D^*_N\Gamma(\bar{x},\bar{y})(\bar{y}^*):=\Limsup_{\begin{subarray}{c} (x,y)\to(\bar{x},\bar{y})\\y^*\stackrel{\mathit{w}^*}{\to}\bar{y}^*\\\varepsilon\downarrow 0\end{subarray}}\hat{D}^*_\varepsilon\Gamma(x,y)(y^*).
$$
If both $E$ and $F$ are Asplund spaces, we have
$$
D^*_N \Gamma(\bar{x},\bar{y})(\bar{y}^*)=\Limsup_{\begin{subarray}{c}(x,y)\to(\bar{x},\bar{y})\\y^*\stackrel{\mathit{w}^*}{\to}\bar{y}^*\end{subarray}}\hat{D}^*\Gamma(x,y)(y^*).
$$

\subsection{Subgradients of Marginal Functions}\label{subseq1}

Here we present a result on subdifferentiation of a general class of {\em marginal functions} in variational analysis, which is instrumental for the subsequent subdifferentiation of the value functions in both deterministic and stochastic DP models of our consideration in what follows.

Given an extended-real-valued function $\varphi\colon E\times F\to\overline{\RR}$ and a multifunction $\Gamma\colon E\rightrightarrows F$ between Banach spaces, the corresponding {\em marginal function} is introduced in the form
\begin{equation}\label{eq1}
\mu(x):=\inf_{y\in\Gamma(x)}\varphi(x,y),\quad x\in E,
\end{equation}
while the associated {\em argminimum multifunction} $G\colon E\rightrightarrows F$ is defined by
\begin{equation}\label{arg}
G(x):=\left\{y\in\Gamma(x)\;\big|\;\varphi(x,y)=\mu(x)\right\},\quad x\in E.
\end{equation}
The marginal function \eqref{eq1} belongs to a general class of extended-real-valued functions, which appear in a broad spectrum of problems in mathematics and its applications that may not be even related to optimization; see \cite{mo06,mo18} for more discussions. On the other hand, we can treat \eqref{eq1} as the ({\em optimal}) {\em value function} of the {\em parametric optimization} problem
\begin{equation*}
\mbox{minimize }\;\ph(x,y)\;\mbox{ subject to }\;y\in\Gamma(x)
\end{equation*}
for which the argminimum multifunction \eqref{arg} defines the parameterized set of optimal solutions. This interpretation is important in what follows.

We say that the argminimum multifunction $G:E\rightrightarrows F$ is \textit{inner semicontinuous} at $(\bar{x},\bar{y})\in \mathrm{gph}\,G$ if for every sequence $x_k\to\bar{x}$ there exists a sequence of $y_k\in G(x_k)$ that contains a subsequence converging to $\bar{y}$. This multifunction is said to be \textit{inner semicompact} at $\bar{x}\in E$ if for every sequence $x_k\to\bar{x}$ there is a sequence of $y_k\in G(x_k)$ that contains a convergent subsequence.

Based on \eqref{eq1}, consider now the function $\vartheta:E\times F\to\overline{\RR}$ given by
$$
\vartheta(x,y):=\varphi(x,y)+\delta((x,y);\mathrm{gph}\,\Gamma)\;\mbox{ for all }\;(x,y)\in E\times F,
$$
where $\delta((\cdot,\cdot);\mathrm{gph}\,\Gamma)$ is the indicator function of $\mathrm{gph}\,\Gamma$, i.e., $\delta((x,y);\mathrm{gph}\,\Gamma):=0$ if $(x,y)\in\mathrm{gph}\,\Gamma$ and $\delta((x,y);\mathrm{gph}\,\Gamma):=\infty$ otherwise.\vspace*{0.05in}

Now we present important upper estimates of the limiting subdifferential of \eqref{eq1} is taken from \cite[Theorem 1.108 and Corollary 1.109]{mo06}.

\begin{prop}[\bf subdifferentiation of marginal functions]\label{prop1}
Let the marginal function \eqref{eq1} be finite at $\bar{x}\in E$ with $G(\bar{x})\ne\emp$, and let both spaces $E$ and $F$ be Banach. The following assertions hold:

{\bf(i)} If $G$ is inner semicontinuous at $(\bar{x},\bar{y})\in\mathrm{gph}\,G$, then
$$
\partial\mu(\bar{x})\subset\left\{x^*\in E^*\mid(x^*,0)\in\partial\vartheta(\bar{x},\bar{y})\right\}.
$$
If furthermore $\varphi$ is strictly differentiable at this point, then
$$
\partial\mu(\bar{x})\subset\nabla_x\varphi(\bar{x},\bar{y})+D^*_N\Gamma(\bar{x},\bar{y})\big(\nabla_y\varphi(\bar{x},\bar{y})\big).
$$

{\bf(ii)} If $G$ is inner semicompact at $\bar{x}$, the graph of $\Gamma$ is closed at $\bar{x}$, and $\varphi$ is lower semicontinuous at every $(\bar{x},y)$ with $y\in\Gamma(\bar{x})$, then we have
$$
\partial\mu(\bar{x})\subset\left\{x^*\in E^*\;\left|\;(x^*,0)\in\bigcup_{\bar{y}\in G(\bar{x})}\partial\vartheta(\bar{x},\bar{y})\right.\right\}.
$$
\end{prop}

\section{Dynamic Programming in Banach Spaces}\label{sec:DP}

The first subsection of this section is devoted to the formulation of a deterministic model of dynamic programming in general Banach spaces with presenting and discussing the major assumptions on its initial data. In the second subsection we prove the Lipschitz continuity of the value function as well as its strict differentiability under additional assumptions, and then we derive new subdifferential necessary optional conditions for this model.

\subsection{Description of the deterministic DP Model}

Let $\NN$ be the set of nonnegative integers, and let the set of time horizons be indexed by $t=0,1,\dots$. For each $t\in\NN$ denote by $X_t$ an \textit{action space}, which is assumed to be an arbitrary Banach. At each time period the decision-maker knows a \textit{cost function} $u_t:X_t\times X_{t+1}\to\overline{\RR}$ and a multifunction $\Gamma_t:X_t\rightrightarrows X_{t+1}$ describing {\em feasibility constraints}. Then an \textit{admissible program} starting from period $t\in\NN$ with the initial condition $x\in X_t$ is an element $(x_t,x_{t+1},\dots)$ in the product space $\prod_{s=t}^\infty X_s$ satisfying $x_{s+1}\in\Gamma_s(x_s)$ for every $s\ge t$ and $x_t=x$. The set of admissible programs from $t$ with $x_t=x$ is denoted by $\A_t(x)$, which gives rise to a multifunction $\A_t:X_t\rightrightarrows \prod_{s=t+1}^\infty X_s$. Having an initial condition $x\in X_0$, we consider the discrete-time deterministic DP problem on the {\em infinite horizon} described by
\begin{equation}\label{DP}
\begin{aligned}
&\inf\sum_{t\in\NN} u_t(x_t,x_{t+1}) \\
&\text{s.t.\ $x_{t+1}\in \Gamma_t(x_t)$ for each $t\in\NN$, $x_0=x\in X_0$}.
\end{aligned}
\end{equation}
Define the (optimal) \textit{value function} $v_t:X_t\to\overline{\RR}$ by
\begin{equation}\label{val1}
v_t(x):=\inf_{(x_t,x_{t+1},\dots)\in\A_t(x)}\sum_{s=t}^\infty u_s(x_s,x_{s+1}).
\end{equation}
An admissible program $(x_0,x_1,\dots)\in \A_0(x)$ with a given $x\in X_0$ is \textit{optimal} if $v_0(x)$ is finite with $v_0(x)=\sum_{t\in\NN}u_t(x_t,x_{t+1})$. For the primitive $\{X_t,\Gamma_t,u_t\}_{t\in \NN}$ of the model, the following {\em summability condition} on the cost function is in force throughout this section.\:
\begin{assmp}
\label{assmp1}
$\displaystyle\sum_{t\in\NN}\sup_{(x,y)\in\mathrm{gph}\,\Gamma_t}|u_t(x,y)|<\infty$.
\end{assmp}

It follows from the \textit{Bellman principle of optimality} that every optimal program $(x_0,x_1,\dots)\in\A_0(x_0)$ to \eqref{DP} satisfies the equality
\begin{equation}\label{BP}
v_t(x_t)=u_t(x_t,x_{t+1})+v_{t+1}(x_{t+1})\quad\text{for each $t\in\NN$}.
\end{equation}
We can verify by standard arguments that the \textit{Bellman equation}
\begin{equation}\label{BE}
v_t(x)=\inf_{y\in\Gamma_t(x)}\left\{u_t(x,y)+v_{t+1}(y)\right\}\quad\text{for every $x\in X_t$}
\end{equation}
is fulfilled for the value function \eqref{val1}. It shows therefore that the value function \eqref{val1} belongs to the class of marginal functions \eqref{eq1}. This simple observation motivates the introduction of the \textit{policy multifunction} $G_t:X_t\rightrightarrows X_{t+1}$ for \eqref{BE} defined by
\begin{equation}\label{eq2}
G_t(x):=\left\{y\in\Gamma_t(x)\mid v_t(x)=u_t(x,y)+v_{t+1}(y)\right\}.
\end{equation}
Any mapping $\gamma_t:X_t\to X_{t+1}$ satisfying $\gamma_t(x)\in G_t(x)$ for every $x\in X_t$ is called a \textit{policy mapping}. By \eqref{BP} and \eqref{BE}, an admissible program $(x_0,x_1,\dots)\in\A_0(x)$ is optimal if and only if $x_{t+1}\in G_t(x_t)$ for each $t\in\NN$ with $x_0=x$. It follows from the classical Berge's maximum theorem that if $\Gamma_t$ is upper semicontinuous with compact values and $u_t$ is lower semicontinuous on $\mathrm{gph}\,\Gamma_t$, then the value function $v_t$ is lower semicontinuous. If moreover the mappings $\Gamma_t$ and $u_t$ are continuous, then $G_t$ is upper semicontinuous; see, e.g., \cite[Lemma~17.30 and Theorem~17.31]{ab06}.\vspace*{0.05in}

The following crucial viability notions were introduced in our paper \cite{ms18}.

\begin{dfn}[\bf local viability]\label{viab}
Let $G_t\colon X_t\rightrightarrows X_{t+1}$ be a policy multifunction with $G_t(\ox)\ne\emp$ for some $\ox\in X_t$. We say that:

{\bf(i)} $G_t$ is {\sc locally lower viable} around $\ox$ if there exists a neighborhood $U$ of $\ox$ such that $G_t(x)\cap\Gamma_t(x')\ne\emp$ for every $x,x'\in U$.

{\bf(ii)} $G_t$  is {\sc locally upper viable} around $\ox$ if there exists a neighborhood $U$ of $\ox$ such that $G_t(x)\subset\Gamma_t(x')$ for every $x,x'\in U$.
\end{dfn}

\noindent
Note that both local upper and lower viability conditions in Definition \ref{viab} are far-going extensions of the standard interiority condition, which says that there exists a neighborhood $U$ of $\bar{x}\in X_t$ such that for every $x\in U$ we can find $y\in G_t(x)$ so that $(x,y)$ belongs to the interior of $\mathrm{gph}\,\Gamma_t$. In particular, the local lower viability condition allows us to obtain the local Lipschitz continuity of the value function $v_t$. The local upper viability condition is used below to evaluate the generalized gradient of $v_t$ and to derive necessary optimality conditions for problem \eqref{DP} in its terms. Observe that the local upper viability condition holds automatically if $\Gamma_t$ is independent of $x$.

\subsection{Necessary Conditions for Optimality}

To formulate the first theorem, denote by $\partial_x^\circ u_t(\bar{x},y)$ the partial generalized gradient \eqref{gg} of the (Lipschitz) function $u_t(\cdot,y)$ at $\bar{x}\in X_t$ when $y\in X_{t+1}$ is fixed. The notation $\partial_y^\circ u_t(x,\bar{y})$ is similar.

\begin{thm}[\bf Lipschitz continuity and the generalized gradient inclusion for the value function]\label{thm1}
Let $X_t$ be a Banach space for each $t\in\NN$, and let $\ox\in X_t$ be such that $G_t(\ox)\ne\emp$. Assume that the cost function $u_t$ is locally Lipschitzian around $(\ox,y)$ for every $y\in G_t(x)$ with $x$ near $\ox$ and that the policy multifunction $G_t$ is locally lower viable around $\ox$. Then the value function $v_t$ is locally Lipschitzian around $\ox$. If in addition $G_t$ is locally upper viable around $\ox$ and if $u_t$ is regular at $(\bar{x},\bar{y})\in\mathrm{gph}\,G_t$ for some $\oy\in\Gamma_t(\ox)$, then we have the generalized grsdient inclusion
\begin{eqnarray}\label{env1}
\partial^\circ v_t(\bar{x})\subset\partial^\circ_x u_t(\bar{x},\bar{y}).
\end{eqnarray}
\end{thm}
\begin{proof}
Fix $\bar{x}\in X_t$ and by the local lower viability of $G_t$ find a neighborhood $U$ of $\bar{x}$ such that $G_t(x)\cap\Gamma_t(x')\ne\emp$ for every $x,x'\in U$. Picking $y\in G_t(x)\cap\Gamma_t(x')$ for arbitrary points $x,x'\in U$ ensures that
$$
v_t(x)=u_t(x,y)+v_{t+1}(y).
$$
Since $v_t(x')\le u_t(x',y)+v_{t+1}(y)$ by \eqref{BE} and since $u_t$ is locally Lipschitzian with modulus $\ell_t$, we have
$$
v_t(x')-v_t(x)\le u_t(x',y)-u_t(x,y)\le\ell_t\|x'-x\|.
$$
Interchanging the role of $x$ and $x'$ above tells us that
$$
|v_t(x)-v_t(x')|\le\ell_t\|x-x'\|\;\mbox{ whenever }\;x,x'\in U,
$$
and hence the value function $v_t$ is locally Lipschitzian.

Next we justify the generalized gradient inclusion \eqref{env1} under the additional assumptions made. It follows from the local upper viability of $G_t$ that for every $x\in U$ and any given direction $h\in X$ we have $G_t(x)\subset\Gamma_t(x+\theta h)$ when $\th>0$ is sufficiently small. Without loss of generality choose $y\in G_t(x)\subset\Gamma_t(x+\theta h)$ for every $\theta>0$ and thus get
$$
v_t(x)=u_t(x,y)+v_{t+1}(y).
$$
By the principle of optimality in dynamic programming we have
$$
v_t(x+\theta h)\le u_t(x+\theta h,y)+v_{t+1}(y),
$$
\begin{equation}\label{eq5}
\frac{v_t(x+\theta h)-v_t(x)}{\theta}\le\frac{u_t(x+\theta h,y)-u_t(x,y)}{\theta}.
\end{equation}
Passing to the limit in \eqref{eq5} as $\th\dn 0$ gives us
$$
\limsup_\substack{(x,y)\stackrel{\mathrm{gph}\,G_t}{\longrightarrow}(\bar{x},\bar{y})}{\theta\downarrow 0}\frac{u_t(x+\theta h,y)-u_t(x,y)}{\theta} \le\limsup_\substack{(x,y)\to(\bar{x},\bar{y})}{\theta\downarrow 0}\frac{u_t(x+\theta h,y)-u_t(x,y)}{\theta},
$$
which readily implies due to \eqref{dd} that
$$
v_t^\circ(\bar{x};h)\le u_t^\circ(\bar{x},\bar{y};(h,0))=u_t'(\bar{x},\bar{y};(h,0))=(u_t)'_x(\bar{x},\bar{y};h)=(u_t)^\circ_x(\bar{x},\bar{y};h),
$$
where $u_t'(\bar{x},\bar{y};(h,0))$ (resp.\ $u_t^\circ(\bar{x},\bar{y};(h,0))$) denotes the (resp.\ generalized) directional derivative of $u_t$ at $(\bar{x},\bar{y})$ in the direction $(h,0)\in X_t\times X_{t+1}$, and where $(u_t)'_x(\bar{x},\bar{y};h)$ (resp.\ $(u_t)^\circ_x(\bar{x},\bar{y};h)$) stands for the partial (resp.\ generalized) directional derivative of $u_t(\cdot,\bar{y})$ at $\bar{x}$ in the direction $h\in X_t$. The obtained inequality is equivalent to
$$
\max_{x^*\in \partial^\circ v_t(\bar{x})}\langle x^*,h \rangle\le\max_{x^*\in\partial^\circ_x u_t(\bar{x},\bar{y})}\langle x^*,h\rangle \quad\text{for every $h\in X_t$}.
$$
Employing finally the convex separation theorem due to the convexity and the weak$^*$ compactness of the generalized gradient sets above, we arrive at \eqref{env1} and thus complete the proof of the theorem.
\end{proof}

As a consequence of Theorem~\ref{thm1}, we get the following result on the strict differentiability of the value function $v_t$. It is a significant improvement upon the known results in this direction with applications to optimal economic growth models (see, e.g., \cite{alv98,bs79,blv96}), since we remove the convexity assumption and mitigate the interior condition in the Banach space setting. For another assumption that replaces the interiority condition to derive the differentiability of the value function under convexity hypotheses, see \cite{rzs09}.

\begin{cor}[\bf strict differentiability of the value function] \label{cor1}
Assume in the setting of Theorem {\rm\ref{thm1}} that $G_t$ is locally upper viable around $\ox$ and that $u_t(\cdot,\bar{y})$ is strictly differentiable at $\bar{x}\in X_t$ with $(\bar{x},\bar{y})\in\mathrm{gph}\,G_t$. Then $v_t$ is strictly differentiable at $\bar{x}$ and its strict derivative at $\ox$ is calculated by
$$
\nabla v_t(\bar{x})=\nabla_x u_t(\bar{x},\bar{y}).
$$
\end{cor}
\begin{proof}
It immediately follows from the facts \cite{cl83} that any function strictly differentiable at a given point is regular at this point and its generalized gradient reduces to the strict derivative therein.
\end{proof}

The next important result, which is formulated via the limiting subdifferential \eqref{bs}, is a consequence of Proposition~\ref{prop1} and Corollary~\ref{cor1}.

\begin{thm}[\bf limiting subgradient inclusions for the value function]\label{lim-val} Let $X_t$ be a Banach space for each $t\in\NN$, and let $\vartheta_t:X_t\times X_{t+1}\to\overline{\RR}$ be the extended-real-valued function defined by
$$
\vartheta_t(x,y):=u_t(x,y)+v_{t+1}(y)+\delta_{\mathrm{gph}\,\Gamma_t}(x,y).
$$
The following assertions are satisfied:

{\bf(i)} If $G_t$ is inner semicontinuous at $(\bar{x},\bar{y})\in\mathrm{gph}\,G_t$, then
$$
\partial v_t(\bar{x})\subset\left\{x^*\in X_t^*\mid (x^*,0)\in\partial\vartheta_t(\bar{x},\bar{y})\right\}.
$$
If furthermore $u_t$ and $v_t$ are strictly differentiable at the reference points for each $t\in\NN$, then for every policy mapping $\gamma_{t+1}:X_{t+1}\to X_{t+2}$ we have
$$
\nabla v_t(\bar{x})\in\nabla_xu_t(\bar{x},\bar{y})+D^*_N\Gamma_t(\bar{x},\bar{y})\big(\nabla_x u_{t+1}(\bar{y},\gamma_{t+1}(\bar{y})\big).
$$

{\bf(ii)} If $G_t$ is inner semicompact at $\bar{x}$, the graph of $\Gamma_t$ is closed at $\bar{x}$, and $u_t$ is lower semicontinuous at every $(\bar{x},y)$ with $y\in\Gamma_t(\bar{x})$, then
$$
\partial v_t(\bar{x})\subset\left\{x^*\in X_t^*\;\left|\;(x^*,0)\in\bigcup_{\bar{y}\in G_t(\bar{x})}\partial\vartheta_t(\bar{x},\bar{y})\right.\right\}.
$$
\end{thm}

The next corollary is in fact a specification of Theorem~\ref{lim-val}.

\begin{cor}[\bf limiting subgradients of the value function under interiority assumptions]\label{lim-val1} The following assertions hold:

{\bf(i)} Assume that $G_t$ is inner semicontinuous at $(\bar{x},\bar{y})\in\mathrm{gph}\,G_t$ and that $(\bar{x},\bar{y})$ is an interior point of $\mathrm{gph}\,\Gamma_t$ at which $u_t$ is strictly differentiable. Then we have the subdifferential inclusion
$$
\partial v_t(\bar{x})\subset\left\{x^*\in X_t^*\mid(x^*,0)\in\nabla u_t(\bar{x},\bar{y})+\big(\{0\}\times\partial v_{t+1}(\bar{y})\big)\right\}.
$$

{\bf(ii)} Assume that $G_t$ is inner semicompact at $\bar{x}\in E_t$, that every $(\bar{x},y)$ with $y\in\Gamma_t(\bar{x})$ belongs to the interior of $\mathrm{gph}\,\Gamma_t$, that the graph of $\Gamma_t$ is closed around $\bar{x}$, that $u_t$ is strictly differentiable at every $(\bar{x},y)$ with $y\in G_t(\bar{x})$, and that $u_t$ is lower semicontinuous at every $(\bar{x},y)$ with $y\in\Gamma_t(\bar{x})$. Then we have
$$
\partial v_t(\bar{x})\subset\left\{ x^*\in X_t^*\;\left|\;(x^*,0)\in\bigcup_{\bar{y}\in G_t(\bar{x})}\Big(\nabla u_t(\bar{x},\bar{y})+\big(\{0\}\times\partial v_{t+1}(\bar{y})\big)\Big)\right.\right\}.
$$
\end{cor}
\begin{proof}
To verify (i), observe that since $(\bar{x},\bar{y})\in\mathrm{gph}\,G_t$ is an interior point, the indicator function $\delta((\cdot,\cdot);\mathrm{gph}\,\Gamma_t)$ has the strict derivative $0$ at $(\bar{x},\bar{y})$. It follows from Theorem~\ref{lim-val}, or directly from \cite[Proposition 1.107]{mo06}, that
$$
\partial\vartheta_t(\bar{x},\bar{y})=\nabla u_t(\bar{x},\bar{y})+\big(\{0\}\times\partial v_{t+1}(\bar{y})\big),
$$
which justifies assertion (i). Then (ii) immediately follows from (i) due to the fact that $\vartheta_t(\bar{x},\cdot)$ is strictly differentiable on $G_t(\bar{x})$ in this case.
\end{proof}

The last result of this section provides a necessary optimality condition in the DP problem \eqref{DP} formulated in the form of the {\em Euler inclusion} and the construction of the generalized normal cone defined in \eqref{cnc}.

\begin{thm}[\bf Euler inclusion for the deterministic DP model]\label{EI}
Let $X_t$ be a Banach space for each $t\in\NN$, and let $\gamma_{t+1}:X_{t+1}\to X_{t+2}$ be a policy mapping. In addition to the assumptions of Theorem~{\rm\ref{thm1}}, suppose that $u_{t+1}$ is regular at $(\bar{y},\gamma_{t+1}(\bar{y}))$, and that $G_{t+1}$ is locally upper viable around $\bar{y}$. Then we have the following Euler inclusion:
\begin{equation}\label{sei}
0\in\partial^\circ_y u_t(\bar{x},\bar{y})+\partial^\circ_x u_{t+1}\big(\bar{y},\gamma_{t+1}(\bar{y})\big)+N^\circ\big(\bar{y};\Gamma_t(\bar{x})\big).
\end{equation}
\end{thm}
\begin{proof}
Since $\oy$ is a local optimal solution to the constrained minimization problem in \eqref{BE}, we have from \cite[Corollary to Proposition~2.4.3]{cl83} that
$$
0\in\partial^\circ_y\big(u_t(\bar{x},\bar{y})+v_{t+1}(\bar{y})\big)+N^\circ\big(\bar{y};\Gamma_t(\bar{x})\big),
$$
which implies by the calculus rules from \cite[Proposition~2.3.1 and Proposition~2.3.3]{cl83} the validity of the inclusion
$$
\partial^\circ_y\big(u_t(\bar{x},\bar{y})+v_{t+1}(\bar{y})\big)\subset\partial^\circ_y u_t(\bar{x},\bar{y})+\partial^\circ v_{t+1}(\oy).
$$
Taking now any policy mapping $\gamma_{t+1}$ and using Theorem~\ref{thm1} tell us that
\begin{equation}\label{eq3}
\partial^\circ v_{t+1}(\bar{y})\subset\partial^\circ_x u_{t+1}\big(\bar{y},\gamma_{t+1}(\bar{y})\big).
\end{equation}
Combining the latter with the inclusions above, we arrive at \eqref{sei}.
\end{proof}

\section{Stochastic Dynamic Programming}

In this section we develop the stochastic dynamic programming model of our main interest in the paper, establish desired properties of the value function, and derive necessary optimality conditions for this model in terms of the novel stochastic Euler equation. The section is split into four subsections that present, respectively, the functional framework of our model, subdifferentiation of integral functionals, the description of the stochastic DP model, and necessary conditions for optimal strategies.

\subsection{$L^p$ Spaces on Banach Spaces}
Let $(\Omega,\Sigma,\mu)$ be a finite measure space. If $1\le p<\infty$, then $L^p(\mu,E)$ stands for the space of all the $E$-valued Bochner integrable mappings $f$ of the $\mu$-equivalence class defined on $\Omega$ with $\int\|f\|^pd\mu<\infty$, where the norm $\|\cdot\|_p$ is given by $\|f\|_p:=(\int\|f(\omega)\|^pd\mu)^{1/p}$. For $p=\infty$ the notation $L^\infty(\mu,E)$ stands for the space of all the $E$-valued Bochner integrable mappings on $\Omega$ of the $\mu$-equivalence class that are essentially bounded with the norm $\|f\|_\infty:=\mathrm{ess\,sup}_{\omega\in\Omega}\|f(\omega)\|$. If $\Sigma$ is countably generated and $E$ is separable, then $L^p(\mu,E)$ is separable for each $1\le p<\infty$; see \cite[Theorem~2.119]{fl07}.

Recall that a mapping $f:\Omega\to E^*$ is \textit{$\mathit{w}^*\!$-scalarly measurable} if for every $x\in E$ the scalar function $\langle f(\cdot),x\rangle:\Omega\to\RR$ defined by $\omega\mapsto\langle f(\omega),x\rangle$ is measurable. Taking $1\le p\le\infty$, denote by $L^p_{\textit{w}^*}(\mu,E^*)$ the space of $E^*$-valued and $\mathit{w}^*\!$-scalarly measurable mappings of the $\mu$-equivalence class on $\Omega$ such that $\|f(\cdot)\|\in L^p(\mu)$ with the norm $\|f\|_p:=(\int\|f(\omega)\|^pd\mu)^{1/p}$. We know that for each $1\le p<\infty$ the dual space of $L^p(\mu,E)$ is given by $L^q_{\textit{w}^*}(\mu,E^*)$ with the conjugate index $q$ for $p$ such that $1/p+1/q=1$ whenever $E$ is separable, where the dual system is defined by $\langle f,g\rangle:=\int\langle f(\omega),g(\omega) \rangle d\mu$ with $f\in L^q_{\textit{w}^*}(\mu,E^*)$ and $g\in L^p(\mu,E)$; see \cite[Theorem~2.112]{fl07}. Since the strong measurability, a defining property of Bochner integrability in $L^p(\mu,E^*)$, implies the $\mathit{w}^*\!$-scalar measurability, it is evident that $L^p(\mu,E^*)\subset L^p_{\textit{w}^*}(\mu,E^*)$ for each index $1\le p\le\infty$.

The \textit{Radon--\hspace{0pt}Nikodym property} (\textit{RNP}) of Banach space $E$ with respect to a finite measure space $(\Omega,\Sigma,\mu)$ postulates that for every $\mu$-\hspace{0pt}continuous vector measure $\nu:\Sigma\to E$ of bounded variation there exists $f\in L^1(\mu,E)$ such that $\nu(A)=\int_Afd\mu$ whenever $A\in\Sigma$. When the space $E$ enjoys the RNP with respect to every finite measure space, it is simply said to have the RNP. Given $1\le p<\infty$ and its conjugate index $q$, the dual space of $L^p(\mu,E)$ is identified with $L^q(\mu,E^*)$ if and only if $E^*$ has the RNP with respect to $(\Omega,\Sigma,\mu)$, where the duality is given by $\langle f,g\rangle:=\int\langle f(\omega),g(\omega)\rangle d\mu$ for $f\in L^q(\mu,E^*)$ and $g\in L^p(\mu,E)$; see \cite[Theorem~IV.1.1]{du77}. Recall finally that $E$ is an {\em Asplund space} (i.e., a Banach space for which any separable subspace has a separable dual) if and only if $E^*$ enjoys the RNP; see \cite[Theorem~5.2.12]{bo83}. This implies that $L^p(\mu,E^*)$ agrees with $L^p_{\textit{w}^*}(\mu,E^*)$ whenever $E^*$ is separable, which is the case when $E$ is Asplund. Thus $L^p(\mu,E)$ is reflexive with $L^p(\mu,E)^*=L^q(\mu,E^*)$ for every $1<p<\infty$ whenever $E$ is an Asplund space.

\subsection{Subdifferentials of Integral Functionals}\label{sec:sub-int}

Denote by $\B(E)$ the Borel $\sigma$-algebra of $E$ with respect its the norm topology, and let $\varphi:E\times\Omega\to\overline{\RR}$ be a $\B(E)\otimes\Sigma$-measurable integrand. The integral functional under investigation is $I_\varphi:L^p(\mu,E)\to\overline{\RR}$ defined by
$$
I_\varphi(f):=\int_\Omega\varphi(f(\omega),\omega)d\mu.
$$
For $1\le p<\infty$ and $p=\infty$ we impose the following Lipschitz properties of the integrand $\varphi(x,\omega)$ with respect to the first variable, respectively.

\begin{description}
\item[$\mathbf{(H_1)}$] There exists a function $k\in L^q(\mu)$ such that
$$
|\varphi(x,\omega)-\varphi(y,\omega)|\le k(\omega)\|x-y\|\quad\text{for every $x,y\in E$ and $\omega\in\Omega$}.
$$
\item[$\mathbf{(H_2)}$]
Let $\bar{f}\in L^\infty(\mu,E)$ be a point at which $I_\varphi$ is finite. There exist a number $\varepsilon>0$ and a function $k\in L^1(\mu)$ such that
\begin{equation*}
\begin{aligned}
|\varphi(x,\omega)-\varphi(y,\omega)|\le k(\omega)\|x-y\|\quad\text{for every $x,y\in\bar{f}(\omega)+\varepsilon\BB$}\\
\text{and $\omega\in \Omega$},
\end{aligned}
\end{equation*}
where $\BB$ stands for the closed unit ball in $E$.
\end{description}

The next result taken from \citet[Theorems~2.7.3 and 2.7.5]{cl83} is a Lipschitzian extension of the subdifferential formula established in \cite{il72} for the case where $\varphi$ is a normal convex integrand with $p=\infty$.

\begin{prop}[\bf generalized gradients of integral functionals]\label{prop2} Let $E$ be a separable Banach space, and let $\bar{f}\in L^p(\mu,E)$ be a point at which $I_\varphi$ is finite. Then the following holds:

{\bf(i)} Under assumption $\mathrm{(H_1)}$ for $1\le p<\infty$ we have
$$
\partial^\circ I_\varphi(\bar{f})\subset\left\{g\in L^q_{\mathit{w}^*}(\mu,E^*)\mid g(\omega)\in\partial^\circ_x\varphi\big(\bar{f}(\omega),\omega\big)\,\text{a.e.}\ \omega\in\Omega\right\}.
$$

{\bf(ii)} Under assumption $\mathrm{(H_2)}$ for $p=\infty$ we have
$$
\partial^\circ I_\varphi(\bar{f})\subset\left\{g\in L^1_{\mathit{w}^*}(\mu,E^*)\mid g(\omega)\in\partial^\circ_x\varphi\big(\bar{f}(\omega),\omega\big)\,\text{a.e.}\ \omega\in\Omega \right\}.
$$
If furthermore $\varphi(\cdot,\omega)$ is regular at $\bar{f}(\omega)$ for every $\omega\in\Omega$, then $I_\varphi$ is also regular at $\bar{f}$ and the above inclusions hold as equality.
\end{prop}

The following result taken from \cite[Theorem 3.2]{ch09} is a significant extension of Proposition~\ref{prop2} for the case $p=1$, where the Lipschitz condition $\mathrm{(H_1)}$ is not required under the nonatomicity of the measure space.

\begin{prop}[\bf limiting subgradients of integral functionals]\label{prop3}
Let $E$ be a separable Banach space, and let $(\Omega,\Sigma,\mu)$ be a nonatomic finite measure space. If $I_\varphi$ is finite at $\bar{f}\in L^1(\mu,E)$, then we have the inclusion
\begin{equation}\label{ch}
\partial I_\varphi(\bar{f})\subset\left\{g\in L^\infty_{\mathit{w}^*}(\mu,E^*)\mid g(\omega)\in\partial_x\varphi(\bar{f}(\omega),\omega)\,\text{a.e.}\ \omega\in\Omega\right\}.
\end{equation}
If furthermore $\varphi(\cdot,\omega)$ is regular at $\bar{f}(\omega)$ for every $\omega\in\Omega$, then $I_\varphi$ is also regular at $\bar{f}$ and the above inclusion holds as equality.
\end{prop}

\noindent
We refer the reader to \cite{gi17,ms18} for more results and discussions on subdifferentiation of integral functionals and, in particular, comparison between the formulas presented in Propositions~\ref{prop2} and \ref{prop3}.

Next we present a new result on calculating the generalized normal cone \eqref{cnc} to the sets of measurable selections of multifunctions.

\begin{thm}[\bf generalized normals to sets of measurable selections]\label{thm2}
Let $E$ be a separable Banach space, and let $M:\Omega\rightrightarrows E$ be a closed-valued multifunction with $\mathrm{gph}\,M\in\Sigma\otimes\B(E)$. Define $\M:=\{f\in L^p(\mu,E)\mid f(\omega)\in M(\omega) \text{ a.e.\ $\omega\in\Omega$} \}$ with $1\le p<\infty$. If $\bar{f}\in\M$ and $M(\omega)$ is regular at $\bar{f}(\omega)\in E$ a.e.\ $\omega\in\Omega$, then we have the equality
$$
N^\circ(\bar{f};\M)=\left\{g\in L^q_{\mathit{w}^*}(\mu,E^*)\mid g(\omega)\in N^\circ(\bar{f}(\omega);M(\omega)) \text{ a.e.\ $\omega\in \Omega$}\right\}.
$$
\end{thm}
\begin{proof}
It follows from \cite[Corollary~8.5.2]{af90} that the generalized tangent cone $T^\circ(\bar{f};\M)$ has the following representation:
\begin{equation}
\label{eq:tan}
T^\circ(\bar{f};\M)=\left\{h\in L^p(\mu,E)\mid h(\omega)\in T^\circ\big(\bar{f}(\omega);M(\omega)\big) \text{ a.e.\ $\omega\in\Omega$} \right\}.
\end{equation}
Take any $g\in N^\circ(\bar{f};\M)$. Suppose that there exist $h\in T^\circ(\bar{f};\M)$ and a set $A\in\Sigma$ with positive measure such that $\langle g(\omega),h(\omega)\rangle>0$ on $A$. Define $\tilde{h}\in L^p(\mu,E)$ by $\tilde{h}(\omega):=h(\omega)$ if $\omega\in A$ and by $\tilde{h}(\omega):=0$ otherwise. Then $\tilde{h}(\omega)\in T^\circ(\bar{f}(\omega);M(\omega))$ a.e.\ $\omega\in\Omega$, and hence $\tilde{h}\in T^\circ(\bar{f};\M)$. This means that $0<\int_A\langle g(\omega),h(\omega)\rangle d\mu=\langle g,\tilde{h} \rangle\le 0$, which contradicts the fact that $g$ belongs to $N^\circ(\bar{f};\M)$. Therefore $\langle g(\omega),h(\omega)\rangle\le 0$ a.e.\ $\omega\in \Omega$ for every $h\in L^p(\mu,E)$ with $h(\omega)\in T^\circ(\bar{f}(\omega);M(\omega))$ The latter yields $g(\omega)\in N^\circ(\bar{f}(\omega);M(\omega))$ a.e.\ $\omega\in\Omega$. The converse inclusion immediately follows from \eqref{eq:tan}.
\end{proof}

\subsection{Description of the Stochastic DP Model}

Now we are ready to describe the nonstationary stochastic DP model of our study in this paper. Our approach is based on the deterministic reduction outlined in \cite{pa94,ta92} for $L^\infty$ spaces.

Let $(\Omega,\Sigma,\mu)$ be a complete probability space, where $\Omega$ is a sample space, let $\Sigma$ is a $\sigma$-\hspace{0pt}algebra of subsets of $\Omega$, and let $\mu$ is a complete probability measure on $\Sigma$. By $\{\Sigma_t \}_{t\in\NN}$ we denote a filtration: $\Sigma_t\subset\Sigma_{t+1}\subset\cdots$ for each $t\in\NN$ with $\Sigma_t$ being a complete sub-$\sigma$-\hspace{0pt}algebra of $\Sigma$ such that $\bigvee_{t\in\NN}\Sigma_t=\Sigma$, where $\bigvee_{t\in\NN}\Sigma_t$ stands for the $\sigma$-algebra generated by $\bigcup_{t\in\NN}\Sigma_t$ and $\Sigma_t$ is the information system available to the decision-maker up to the period $t\in\NN$. Having a sequence of Banach spaces $\{E_t\}_{t\in\NN}$, for every $1\le p\le \infty$ denote by $L^p(\Sigma_t,\mu;E_{t+1})$ the space of $E_{t+1}$-\hspace{0pt}valued, $\Sigma_t$-measurable, and Bochner integrable mappings $f$ on $\Omega$ with $\int\|f(\omega)\|^pd\mu<\infty$ and by $L^p_{\mathit{w}^*}(\Sigma_t,\mu;E_t^*)$ the space of $E_t^*$-\hspace{0pt}valued and $\mathit{w}^*\!$-scalarly measurable mappings $f$ on $\Omega$ with respect to $\Sigma_t$ such that $\|f(\omega)\|\in L^p(\mu)$.

The primitive of the model is described by a filtration $\{\Sigma_t\}_{t\in\NN}$ of a probability space $(\Omega,\Sigma,\mu)$, a sequence $\{E_t\}_{t\in\NN}$ of Banach spaces, a random multifunction $\Phi_t:E_t\times\Omega\rightrightarrows E_{t+1}$, and a random cost function $\varphi_t:E_t\times E_{t+1}\times\Omega\to\overline{\RR}$. Given an initial condition $f_0\in L^p(\Sigma_0,\mu;E_0)$, an adapted stochastic process $\{f_t\}_{t\in\NN}$ with $f_{t+1}\in L^p(\Sigma_t,\mu;E_{t+1})$ and $f_{t+1}(\omega)\in\Phi_{t}(f_t(\omega),\omega)$ a.e.\ $\omega\in\Omega$ for each $t\in\NN$ is called an \textit{admissible program}. The stochastic DP problem under investigation is defined as follows:
\begin{equation}\label{SDP}
\begin{aligned}
&\inf{}\sum_{t\in \NN}\int_\Omega\varphi_t(f_t(\omega),f_{t+1}(\omega),\omega)d\mu\\
&\text{s.t. } f_{t+1}(\omega)\in\Phi_t(f_t(\omega),\omega) \text{ a.e.\ $\omega\in\Omega$},\\
&\hspace{0.8cm}f_{t+1}\in L^p(\Sigma_t,\mu;E_{t+1}) \text{ for each $t\in\NN$},\\
&\hspace{0.8cm}f_0\in L^p(\Sigma_0,\mu;E_0).
\end{aligned}
\end{equation}
We impose the following standing requirements on the initial data of \eqref{SDP}.

\begin{assmp}\label{assmp2} For the stochastic DP model \eqref{SDP}, suppose that:

{\bf(i)} $\mathrm{gph}\,\Phi_t$ belongs to $\B(E_t)\otimes\Sigma_t\otimes\B(E_{t+1})$.

{\bf(ii)} $\varphi_t:E_t\times E_{t+1}\times\Omega\to\overline{\RR}$ is $\B(E_t)\otimes\B(E_{t+1})\otimes\Sigma_t$-measurable.

{\bf(iii)} There exists a function $\alpha\in L^1(\mu)$ such that
$$
\sum_{t\in\NN}\sup_{(x,y)\in\mathrm{gph}\,\Phi_t(\cdot,\cdot,\omega)}|\varphi_t(x,y,\omega)|\le\alpha(\omega)\quad\text{for every $\omega\in\Omega$}.
$$
\end{assmp}

\noindent
Observe that due to the measurability condition (ii) in Assumption~\ref{assmp2}, for any $(f,g)\in L^p(\Sigma_{t-1},\mu;E_t)\times L^p(\Sigma_t,\mu;E_{t+1})$ we can easily deduce that the random cost function $\omega\mapsto \varphi_t(f(\omega),g(\omega),\omega)$ and the random multifunction $\omega\mapsto\Phi_t(f(\omega),g(\omega),\omega)$ are $\Sigma_t$-measurable.

To transform \eqref{SDP} into the deterministic DP problem of type \eqref{DP} in the Banach space setting, it is sufficient to make the notational change by denoting $X_{t}:=L^p(\Sigma_{t-1},\mu;E_{t})$ for each $t\in\NN$ with $\Sigma_{-1}:\equiv\Sigma_0$ and then defining the multifunction $\Gamma_t:L^p(\Sigma_{t-1},\mu;E_t)\rightrightarrows L^p(\Sigma_t,\mu;E_{t+1})$ by
\[
\Gamma_t(f):=\left\{g\in L^p(\Sigma_t,\mu;E_{t+1})\mid g(\omega)\in\Phi_t(f(\omega),\omega) \text{ a.e.\ $\omega\in \Omega$}\right\}
\]
and the cost function $u_t:L^p(\Sigma_{t-1},\mu;E_t)\times L^p(\Sigma_t,\mu;E_{t+1})\to \overline{\RR}$ by
\begin{equation}\label{ut}
u_t(f,g):=\int_\Omega\varphi_t(f(\omega),g(\omega),\omega)d\mu.
\end{equation}
In this way we get the deterministic DP model in Banach spaces written as
\begin{equation}\label{DP-deter}
\begin{aligned}
&\inf\sum_{t\in\NN}u_t\big(f_t,f_{t+1}\big)\\
& \text{s.t.\ $f_{t+1}\in\Gamma_t(f_t)$ for each $t\in\NN$, $f_0\in L^p(\Sigma_0,\mu;E_{0})$}
\end{aligned}
\end{equation}
with the value function $v_t:L^p(\Sigma_t,\mu;E_{t+1})\to\overline{\RR}$ given by
\begin{equation}\label{svf}
v_t(f):=\inf_{(f_t,f_{t+1},\dots)\in\A_t(f)}\sum_{s=t}^\infty u_s(f_s,f_{s+1}).
\end{equation}
While the optimality of admissible programs is \eqref{DP-deter} is formulated, the Bellman principle of optimality \eqref{BP} and the Bellman equation \eqref{BE} are valid in this framework being discussed in Section~\ref{sec:DP} together with the policy multifunction $G_t:L^p(\Sigma_{t-1},\mu;E_t)\rightrightarrows L^p(\Sigma_{t},\mu;E_{t+1})$ and the policy mapping $\gamma_t:L^p(\Sigma_{t-1},\mu;E_t)\to L^p(\Sigma_{t},\mu;E_{t+1})$ that are defined similarly to \eqref{eq2}. All of this being combined with the Leibniz-type rules for the subdifferentiation of integral functions given in Subsection~\ref{sec:sub-int} allows us to derive necessary optimality conditions for the stochastic DP \eqref{SDP} in the next subsections.

\subsection{Necessary Optimality Conditions for Stochastic DP}

In this subsection we establish necessary optimality conditions for the stochastic DP problem \eqref{SDP} while concentrating on the conditions expressed in terms of the generalized gradient \eqref{gg}. Extensions of the obtained results to the case of the limiting subdifferential \eqref{bs} is an open question due to the absence of the Asplund property for the space $L^1(\mu,E)$ in Proposition~\ref{prop3}; see Section~\ref{conc} for more discussions and references.

To proceed further, we need the following Lipschitzian assumption on the cost function $\varphi(x,y,\omega)$ in \eqref{SDP} with respect to first two variables:

\begin{assmp}\label{assmp3}
There exists a function $k_t\in L^q(\mu)$ such that
$$
|\varphi_t(x,y,\omega)-\varphi_t(x',y',\omega)|\le k_t(\omega)(\| x-x'\|+\| y-y'\|)
$$
for every $x,x'\in E_t$, $y,y'\in E_{t+1}$, and $\omega\in\Omega$.
\end{assmp}

\noindent
It is easy to see that Assumption~\ref{assmp3} guarantees that the function $u_t$ in \eqref{ut} is Lipschitz continuous of rank $\int k_td\mu$. If moreover $\varphi_t(\cdot,\cdot,\omega)$ is regular at every $(x,y)\in E_t\times E_{t+1}$ for every $\omega\in \Omega$, then $u_t$ is also regular at every $(f,g)\in L^p(\Sigma_{t-1},\mu;E_t)\times L^p(\Sigma_t,\mu;E_{t+1})$.\vspace*{0.05in}

The next theorem presents a necessary optimality condition for the stochastic DP model \eqref{SDP} given in terms of the generalized gradient inclusion for the value functions with justifying its Lipschitz continuity as well as strict differentiability under an additional assumption. Recall that $G_t$ stands below for the policy multifunction introduced in \eqref{eq2}, and that the viability conditions are taken from Definition~\ref{viab}.

\begin{thm}[\bf generalized gradient inclusion for the value function in stochastic DP]\label{mt1}
Let $E_t$ be a separable Banach space for each $t\in\NN$, and let $(\bar{f},\bar{g})\in L^p(\Sigma_{t-1},\mu;E_{t})\times L^p(\Sigma_{t},\mu;E_{t+1})$ with $1\le p\le\infty$ be such that $\bar{g}\in G_t(\bar{f})$. Assume that $G_t$ is locally upper viable around $\bar{f}$ and $\varphi_t(\cdot,\cdot,\omega)$ is regular at every $(x,y)\in E_t\times E_{t+1}$ for every $\omega\in\Omega$. Under Assumptions~{\rm\ref{assmp2}} and {\rm\ref{assmp3}} we have that the value function $v_t$ from \eqref{svf} is Lipschitz continuous on $E_t\times E_{t+1}$ and satisfies the generalized gradient inclusion
$$
\partial^\circ v_t(\bar{f})\subset\left\{h\in L^q_{\mathit{w}^*}(\Sigma_{t},\mu;E_t^*)\mid h(\omega)\in\partial^\circ_x\varphi_t\big(\bar{f}(\omega),\bar{g}(\omega),\omega\big) \ \text{a.e.}\ \omega\in\Omega\right\}.
$$
If furthermore $\varphi_t(\cdot,\cdot,\omega)$ is strictly differentiable at every point $(x,y)\in E_t\times E_{t+1}$ a.e.\ $\omega\in\Omega$, then $v_t$ is strictly differentiable at $\bar{f}$ with
$$
\nabla v_t(\bar{f})=\nabla_x \varphi_t\big(\bar{f}(\cdot),\bar{g}(\cdot),\cdot\big)\in L^q_{\mathit{w}^*}(\Sigma_{t},\mu;E_t^*).
$$
\end{thm}
\begin{proof}
Define $\bar{\varphi}_t:E_t\times \Omega\to \overline{\RR}$ by $\bar{\varphi}_t(x,\omega):=\varphi_t(x,\bar{g}(\omega),\omega)$ and consider the integral functional $I_{\bar{\varphi}_t}:L^p(\Sigma_{t-1},\mu;E_t)\to\overline{\RR}$ given by
$$
I_{\bar{\varphi}_t}(f):=\int\bar{\varphi}_t(f(\omega),\omega)d\mu.
$$
Since $\bar{f}$ and $\bar{\varphi}_t(\bar{f}(\cdot),\cdot)$ are $\Sigma_t$-measurable, it follows from the above deterministic reduction with the usage of Theorem~\ref{thm1} and Proposition~\ref{prop2} that $v_t$ is Lipschitz continuous on $E_t\times E_{t+1}$ and satisfies the relationships
\begin{equation*}
\begin{aligned}
\partial^\circ v_t(\bar{f})
&\subset\partial^\circ_x u_t(\bar{f},\bar{g})=\partial^\circ I_{\bar{\varphi}_t}(\bar{f})\\
&\subset\left\{ h\in L^q_{\mathit{w}^*}(\Sigma_{t},\mu;E_t^*)\mid h(\omega)\in\partial^\circ_x \bar{\varphi}_t(\bar{f}(\omega),\omega) \ \text{a.e.}\ \omega\in\Omega\right\}\\
&=\left\{h\in L^q_{\mathit{w}^*}(\Sigma_{t},\mu;E_t^*)\mid h(\omega)\in\partial^\circ_x\varphi_t(\bar{f}(\omega),\bar{g}(\omega),\omega) \ \text{a.e.}\ \omega\in\Omega\right\},
\end{aligned}
\end{equation*}
which bring us to the claim inclusion for $\partial^\circ v_t(\bar{f})$. The strict differentiability of $v_t$ follows from the above procedure with the usage of Corollary~\ref{cor1}.
\end{proof}

Finally in this section, we arrive at the following state-dependent stochastic Euler inclusion obtained in terms of the given data of \eqref{SDP}. Observe that our necessary optimality condition does not involve any integration operation, i.e., expectation in the probabilistic sense. This is a sharp contrast with the results obtained in \cite{pa94,ta92}.

\begin{thm}[\bf stochastic Euler equation]\label{mt2}
Let $E_t$ be a separable Banach space for each $t\in\NN$, and let the pair $(\bar{f},\bar{g})\in L^p(\Sigma_{t-1},\mu;E_{t})\times L^p(\Sigma_{t},\mu;E_{t+1})$ with $1\le p\le\infty$ be such that $\bar{g}\in G_t(\bar{f})$. Assume that the set $\Phi_t(\bar{f}(\omega),\omega)$ is regular at $\bar{g}(\omega)$ and closed a.e.\ $\omega\in\Omega$, that the multifunction $G_t$ is locally upper viable around $\bar{f}$, and that the cost function $\varphi_t(\cdot,\cdot,\omega)$ is regular at every $(x,y)\in E_t\times E_{t+1}$ for each $\omega\in\Omega$. Then under Assumptions~{\rm\ref{assmp2}} and {\rm\ref{assmp3}} we have the stochastic Euler inclusion
\begin{equation}\label{sei1}
\begin{aligned}
0\in\partial^\circ_y\varphi_t\big(\bar{f}(\omega),\bar{g}(\omega),\omega\big)
&+\partial^\circ_x\varphi_{t+1}\big(\bar{g}(\omega),\gamma_{t+1}(\bar{g})(\omega),\omega\big)\\
&+N^\circ\big(\bar{g}(\omega);\Phi_t(\bar{f}(\omega),\omega)\big)\quad \text{a.e.\ $\omega\in\Omega$}.
\end{aligned}
\end{equation}
\end{thm}
\begin{proof}
Consider $\bar{\varphi}_t:E_{t+1}\times\Omega\to\overline{\RR}$ given by $\bar{\varphi}_t(y,\omega):=\varphi_t(\bar{f}(\omega),y,\omega)$ and define the integral functional $I_{\bar{\varphi}_t}:L^p(\Sigma_{t},\mu;E_{t+1})\to\overline{\RR}$ by
$$
I_{\bar{\varphi}_t}(g):=\int\bar{\varphi}_t\big(g(\omega),\omega\big)d\mu.
$$
Then we deduce from Proposition~\ref{prop2} that
\begin{equation*}
\begin{aligned}
\partial^\circ_y u_t(\bar{f},\bar{g})
&=\partial^\circ I_{\bar{\varphi}_t}(\bar{g})\\
&\subset\left\{h\in L^q_{\mathit{w}^*}(\Sigma_{t},\mu;E_{t+1}^*)\mid h(\omega)\in\partial^\circ_x\bar{\varphi}_t\big(\bar{g}(\omega),\omega\big) \ \text{a.e.}\ \omega\in\Omega\right\}\\
&=\left\{h\in L^q_{\mathit{w}^*}(\Sigma_{t},\mu;E_{t+1}^*)\mid h(\omega)\in\partial^\circ_y\varphi_t\big(\bar{f}(\omega),\bar{g}(\omega),\omega\big) \ \text{a.e.}\ \omega\in\Omega\right\}.
\end{aligned}
\end{equation*}
It follows from the proof of Theorem~\ref{mt1} that
\begin{equation*}
\begin{aligned}
&\partial^\circ_x u_{t+1}(\bar{g},\gamma_{t+1}(\bar{g}))\\{}\subset
&\left\{h\in L^q_{\mathit{w}^*}(\Sigma_{t+1},\mu;E_{t+1}^*)\left|\begin{array}{r}h(\omega)\in\partial^\circ_x\varphi_{t+1}\big(\bar{g}(\omega),\gamma_{t+1}(\bar{g})(\omega),\omega\big)\\ \text{a.e.}\ \omega\in\Omega\end{array}\right.\right\}.
\end{aligned}
\end{equation*}
Furthermore, applying Theorem \ref{thm2} leads us to
$$
N^\circ(\bar{g};\Gamma_t(\bar{f}))=\left\{h\in L^q_{\mathit{w}^*}(\Sigma_t,\mu;E_{t+1}^*)\left|\begin{array}{r} h(\omega)\in N^\circ\big(\bar{g}(\omega);\Phi_t(\bar{f}(\omega),\omega)\big)\\ \text{a.e.\ $\omega\in\Omega$}\end{array}\right.\right\}.
$$
Employing finally Theorem~\ref{EI} in our setting ensures the validity of
$$
0\in\partial^\circ_y u_t(\bar{f},\bar{g})+\partial^\circ_x u_{t+1}\big(\bar{g},\gamma_{t+1}(\bar{g})\big)+N^\circ\big(\bar{g};\Gamma_t(\bar{f})\big),
$$
which clearly verifies the stochastic Euler inclusion \eqref{sei1}.
\end{proof}

\section{Concluding Remarks}\label{conc}
We conclude this paper with some comments and open research questions.\vspace*{0.05in}

$\bullet$ To the best of our knowledge, available sum rules for the limiting subdifferential are rather vague in the non-Asplund spaces $L^1(\Sigma_t,\mu;E_{t+1})$ and $L^\infty(\Sigma_t,\mu;E_{t+1})$. It strongly relates to the fact that subdifferential formulas of the type given in Proposition~\ref{prop3} are not currently established in the aforementioned spaces. Deriving such formulas is a major problem of the future research important for its own sake and for the purpose of applications to subdifferentiation of the value functions in stochastic dynamic programming.\vspace*{0.05in}

$\bullet$ It would be very important to extend Proposition~\ref{prop3} to the spaces $L^p(\mu,E)$ with $2\le p<\infty$, which are Asplund. This would open the gate to calculate the limiting subdifferential of the value functions considered above due to the availability of the comprehensive limiting subdifferential sum rules for broad classes of functions defined on Asplund spaces; see \cite{mo06}. In our stochastic DP setting we can apply the limiting subdifferential sum rules on the Asplund space $L^p(\Sigma_t,\mu;E_{t+1})$ with $2\le p<\infty$ whenever $E_t$ is an Asplund space for each $t\in\NN$. Observe to this end that an adequate extension of Theorem~\ref{thm2} to the limiting normal cone is also required. Note that in case of finite-dimensional spaces $E_t$ the desired formula has been recently obtained in \cite{mw18}. This allows us to establish counterparts of our results in such settings.\vspace{0.05in}

$\bullet$ In the derivation of the necessary optimality conditions given in Theorems~\ref{thm1}, \ref{EI}, \ref{mt1}, and \ref{mt2} we impose the regularity assumption on the (random) cost function. Although in most economic applications regularity is an innocuous assumption, especially in the convex settings, it is a rather strong requirement from the viewpoint of variational analysis and generalized differentiation. We plan to significantly relax it in our future research.

\end{document}